\documentclass[11pt,twoside,reqno,psamsfonts]{amsart}

\usepackage[OT1]{fontenc}
\usepackage{type1cm}
\usepackage{amssymb}
\usepackage{esint}
\usepackage[left=2.3cm,top=3cm,right=2.3cm]{geometry}

\numberwithin{equation}{section}

\theoremstyle{plain}
\newtheorem{theorem}{Theorem}[section]
\newtheorem{corollary}[theorem]{Corollary}
\newtheorem{proposition}[theorem]{Proposition}
\newtheorem{lemma}[theorem]{Lemma}

\theoremstyle{remark}
\newtheorem{remark}[theorem]{Remark}
\newtheorem{example}[theorem]{Example}

\theoremstyle{definition}

\newcommand{\BB}{\mathcal{B}}
\newcommand{\CC}{\mathcal{C}}

\newcommand{\LL}{\mathcal{L}}

\newcommand{\QQ}{\mathcal{Q}}

\newcommand{\R}{\mathbb{R}}

\newcommand{\Q}{\mathbb{Q}}

\newcommand{\N}{\mathbb{N}}

\newcommand{\iii}{\mathtt{i}}
\newcommand{\jjj}{\mathtt{j}}
\newcommand{\kkk}{\mathtt{k}}

\newcommand{\eps}{\varepsilon}

\DeclareMathOperator{\dimloc}{dim_{loc}}
\DeclareMathOperator{\udimloc}{\overline{dim}_{loc}}
\DeclareMathOperator{\ldimloc}{\underline{dim}_{loc}}

\DeclareMathOperator{\uld}{\overline{D}_{loc}}
\DeclareMathOperator{\lld}{\underline{D}_{loc}}

\DeclareMathOperator{\udim}{\overline{dim}}
\DeclareMathOperator{\ldim}{\underline{dim}}

\DeclareMathOperator{\dimh}{dim_H}

\DeclareMathOperator{\dimp}{dim_p}

\DeclareMathOperator{\dist}{dist}
\DeclareMathOperator{\diam}{diam}

\DeclareMathOperator{\spt}{spt}

\begin{document}

\title[Local multifractal analysis]{Local multifractal analysis in metric spaces}

\author{Antti K\"aenm\"aki}
\author{Tapio Rajala}
\address{Department of Mathematics and Statistics \\
         P.O. Box 35 (MaD) \\
         FI-40014 University of Jyv\"askyl\"a \\
         Finland}
\email{antti.kaenmaki@jyu.fi}
\email{tapio.m.rajala@jyu.fi}

\author{Ville Suomala}
\address{Department of Mathematical sciences\\
P.O Box 3000\\
FI-90014 University of Oulu\\
Finland }
\email{ville.suomala@oulu.fi}

\thanks{The authors acknowledge the support of the Academy of Finland, projects \#114821, \#126976, \#137528 and \#211229}
\date{\today}

\begin{abstract}
We study the local dimensions and local multifractal properties of measures on doubling metric spaces. Our aim is twofold. On one hand, we show that there are plenty of multifractal type measures in all metric spaces which satisfy only mild regularity conditions. On the other hand, we consider a local spectrum that can be used to gain finer information on the local behaviour of measures than its global counterpart.
\end{abstract}

\maketitle

\section{Introduction}
In multifractal analysis, the interest is in the behaviour of the local dimension map
\[x\mapsto\dimloc(\mu,x)=\lim_{r \downarrow 0} \log\mu(B(x,r))/\log r,\]
for some, often dynamically defined, fractal type measures $\mu$. From the mathematical point of view, the ultimate goal is to understand the size of the level sets
\[E_\alpha=\{x\,:\,\dimloc(\mu,x)=\alpha\}.\]
It is common to say that ``$\mu$ satisfies the
multifractal formalism'' if for all $\alpha\ge0$ the Hausdorff and packing dimensions of $E_\alpha$ are given by the Legendre transform of the $L^q$-spectrum $\tau_q$, that is,
\begin{equation}\label{eq:formalsim}
\dimh(E_\alpha)=\dimp(E_\alpha)=\inf_{q\in\R}\{q\alpha-\tau_q(\mu)\}.
\end{equation}
See Section \ref{sec:not} below for the precise definitions.

Since its origins in physics literature in the 80's (e.g.\ \cite{HentschelProcaccia1983, FrischParisi1985}), the multifractal analysis has gained a lot of interest. For many relevant works related to multifractal formalism, see e.g.\ references in \cite{FengLau2009}. However, it seems that most of the studies take place in Euclidean spaces, or in spaces having a Euclidean type manifold structure.

In this paper, our goal is to study the local dimensions of measures in doubling metric spaces. Perhaps the most classical situation in which the multifractal
formalism is known to hold is the case of self-similar measures
in Euclidean spaces under the strong separation condition; see
e.g.\ \cite{CawleyMauldin1992, Falconer1997}. Our main results can be viewed as a
generalisation of this result into metric spaces, but our method using local versions of the $L^q$-spectrum and dimensions is useful also in the Euclidean setting.

In a general doubling metric space, there are usually no nontrivial self-similar maps, but often there is still a large class of Moran constructions sharing many of the geometric properties of self-similar iterated function systems. We will consider measures on the limit sets of these Moran constructions and investigate the behaviour and multifractality of $\dimloc(\mu,x)$ for these measures.
To determine $\dimloc(\mu,x)$, we consider the local $L^q$-spectrum of $\mu$.
As for the classical (global) spectrum, the definition involves sums of the form $\sum_{B\in\BB}\mu(B)^q$ over packings or partitions of the space $X$. However, in order to make the notion local, only those $B\in\BB$ are taken into account which are ``sufficiently close to $x$''. It turns out that in many cases, the local spectrum gives more precise information on $\dimloc(\mu,x)$ than its global counterpart. The local spectrum was introduced in \cite{KaenmakiRajalaSuomala2012} as a tool to study local homogeneity properties of measures. Although the definition seems very natural, we were not able to track a definition of a local spectrum for measures in the existing literature. In \cite{Barraletal2010}, a local spectrum for functions is defined in order to study their H\"older regularity. After the completion of our work, the paper \cite{Barraletal2012} was made public. The paper deals with local multifractal analysis in Euclidean spaces for functions, measures, and distrib!
 utions.

The paper is organised as follows. In Section \ref{sec:not}, we set up some notation and define the necessary concepts. Further, in Section \ref{sec:par} we consider partitions of the space $X$, and show how the various dimensions and dimension spectra can be calculated using these partitions. We also relate the $L^q$-spectra and dimensions to the local entropy dimensions defined in \cite{KaenmakiRajalaSuomala2012}. Our main results are presented in Section \ref{sec:results}. We first give a series of conditions for Moran constructions in doubling metric spaces and measures defined on their limit sets. Then we show how the local $L^q$-spectrum can be used to calculate the local dimensions of these measures, and finally study their multifractal properties.

\section{Notation and preliminaries}\label{sec:not}

In this paper, we always assume the metric space $(X,d)$ to be \emph{doubling},
meaning that there is a constant $N=N(X)\in\N$, called the \emph{doubling constant} of $X$,
such that any closed ball $B(x,r) = \{ y\in X : d(x,y) \le r \}$ with centre $x \in X$ and radius $r>0$ can be
covered by $N$ balls of radius $r/2$.

For $M>0$ and a ball $B=B(x,r)$, we will use the abbreviation $MB=B(x,Mr)$. For this to make sense, we always assume that the radius and centre of the ball $B$ have been fixed, even if these are not explicitly mentioned.

We call any countable collection $\BB$ of pairwise disjoint closed balls a \emph{packing}. It is called a \emph{packing of $A$} for a subset $A \subset X$ if the centers of the balls of $\BB$ are in the set $A$, and it is a $\delta$-packing (for $\delta > 0$) if all the balls in $\BB$ have radius $\delta$. A $\delta$-packing $\BB$
of $A$ is termed \emph{maximal} if for every $x \in A$ there is $B \in \BB$ so that $B(x,\delta) \cap B \ne \emptyset$. Note that if $\BB$ is a maximal $\delta$-packing of $A$, then $2\BB = \{ 2B : B \in \BB \}$ covers $A$.

The following lemma will be frequently used in this paper.

\begin{lemma} \label{thm:covering_thm}
  For a metric space $X$, the following statements are equivalent:
\begin{enumerate}
  \item $X$ is doubling. \label{covering1}
  \item There are $s>0$ and $c>0$ such that for all $R>r>0$ any ball of radius $R$ can be covered by $c(r/R)^{-s}$ balls of radius $r$. \label{covering2}
  \item There are $s>0$ and $c>0$ such that if $R>r>0$ and $\BB$ is an $r$-packing of a closed ball of radius $R$, then the cardinality of $\BB$ is at most $c(r/R)^{-s}$. \label{covering3}
  \item For every $0<\lambda<1$ there is a constant $M=M(X,\lambda) \in
  \N$, satisfying the following: If $\BB$ is a collection of closed
  balls of radius $\delta > 0$ so that $\lambda\BB$ is pairwise
  disjoint,
  then there are $\delta$-packings $\{ \BB_1,\ldots,\BB_M \}$ so that
  $\BB=\bigcup_{i=1}^M\BB_i$. \label{covering4}
  \item There is $M=M(X) \in \N$ such that if $A\subset X$ and $\delta>0$, then there are $\delta$-packings of $A$, $\BB_1,\ldots,\BB_M$ whose union covers $A$. \label{covering5}
\end{enumerate}
\end{lemma}

The \emph{upper and lower local dimensions of a measure $\mu$ at $x$} are defined by
\begin{align*}
  \udimloc(\mu,x) &= \limsup_{r \downarrow 0} \log\mu(B(x,r))/\log r, \\
  \ldimloc(\mu,x) &= \liminf_{r \downarrow 0} \log\mu(B(x,r))/\log r,
\end{align*}
respectively. If the upper and lower dimensions agree, we call their
mutual value the \emph{local dimension of the measure $\mu$ at $x$}
and write $\dimloc(\mu,x)$ for this common value.
In this article, a \emph{measure} exclusively refers to a nontrivial Borel regular (outer) measure defined on all subsets of $X$ so that bounded sets have finite measure.

For estimating the local dimensions we will use the local $L^q$-dimensions defined in \cite{KaenmakiRajalaSuomala2012}.
Although the local definitions are obtained via their global counterparts in small balls, their behaviour can be quite different; see \cite[Examples 5.1--5.2]{KaenmakiRajalaSuomala2012}.

Let $\mu$ be a measure on $X$, $A\subset X$ a bounded set and $q\in\R$.
The \emph{(global) $L^q$-spectrum of $\mu$ on $A$} is defined by
\begin{equation*}
  \tau_q(\mu,A)=\liminf_{\delta \downarrow 0}
  \frac{\log S_{q}(\mu,A,\delta)}{\log\delta},
\end{equation*}
where
\begin{equation}\label{Sdef}
  S_{q}(\mu,A,\delta) = \sup\Bigl\{ \sum_{B \in \BB} \mu(B)^q :
  \BB \text{ is a $\delta$-packing of } A\cap\spt(\mu) \Bigr\}
\end{equation}
is the \emph{$L^q$-moment sum of $\mu$ on $A$ at the scale $\delta$}. Note that if $q\ge 0$, the definition of $\tau_q(\mu,A)$ does not change if $A\cap\spt(\mu)$ is replaced by $A$ in the right-hand side of \eqref{Sdef}.
If $q \ne 1$, then we define the \emph{(global) $L^q$-dimension of $\mu$ on $A$} by setting
\begin{equation*}
  \dim_q(\mu,A) = \tau_q(\mu,A)/(q-1).
\end{equation*}
We also denote $\tau_q(\mu) = \tau_q(\mu,X)$ and $\dim_q(\mu) = \dim_q(\mu,X)$ provided that $X$ is bounded.

In the case $q = 1$ the above definition makes no sense. Thus we define for every $A\subset X$ with $\mu(A)>0$ the \emph{(global) upper and lower entropy dimensions of $\mu$ on $A$} as
\begin{equation*}
\begin{split}
  \udim_1(\mu,A) &= \limsup_{\delta \downarrow 0} \fint_{A} \frac{\log\mu(B(y,\delta))}{\log\delta}\,d\mu(y), \\
  \ldim_1(\mu,A) &= \liminf_{\delta \downarrow 0} \fint_{A} \frac{\log\mu(B(y,\delta))}{\log\delta}\,d\mu(y),
\end{split}
\end{equation*}
respectively. If they agree, then their common value is denoted by $\dim_1(\mu,A)$.
Here and hereafter, for $A\subset X$ and a $\mu$-measurable $f\colon
X\rightarrow\overline{\R}$, we use the notation $\fint_A f(y)\,d\mu(y) =
\mu(A)^{-1} \int_A f(y)\,d\mu(y)$ whenever the integral is well defined.

From the above global definitions we then derive their local versions.
The \emph{local $L^q$-spectrum of $\mu$ at $x\in\spt(\mu)$} is defined as
\begin{equation*}
  \tau_q(\mu,x) = \lim_{r \downarrow 0} \tau_q(\mu, B(x,r))
\end{equation*}
and the \emph{local $L^q$-dimension of $\mu$ at $x$} as
\begin{equation*}
  \dim_q(\mu,x) = \lim_{r \downarrow 0} \dim_q(\mu,B(x,r)) = \tau_q(\mu,x)/(q-1).
\end{equation*}
Correspondingly, the \emph{local upper and lower entropy dimensions at $x\in\spt(\mu)$} are defined as
\begin{equation*}
\begin{split}
  \udim_1(\mu,x) &= \limsup_{r \downarrow 0} \udim_1(\mu,B(x,r)),\\
  \ldim_1(\mu,x) &= \liminf_{r \downarrow 0} \udim_1(\mu,B(x,r)).
\end{split}
\end{equation*}
For the basic properties of $\dim_q$, we refer to \cite{KaenmakiRajalaSuomala2012}.

The following theorem lists the main relationships between the different local dimensions. Recall that a measure $\mu$ has the density point property, if
\begin{equation*}
  \lim_{r \downarrow 0} \frac{\mu(A \cap B(x,r))}{\mu(B(x,r))} = 1
\end{equation*}
for $\mu$-almost all $x \in A$ whenever $A \subset X$ is $\mu$-measurable.
Note that although in Euclidean spaces the density point property is satisfied for all measures, this is not necessarily the case in doubling metric spaces; see \cite[Example 5.6]{KaenmakiRajalaSuomala2012}.

\begin{theorem} \label{thm:pointwise_dimq}
  If $\mu$ is a measure on a doubling metric space $X$, then
  \begin{equation} \label{eq:local_dimq}
    \lim_{q \downarrow 1} \dim_q(\mu,x) \le \ldimloc(\mu,x) \le \udimloc(\mu,x) \le \lim_{q \uparrow 1} \dim_q(\mu,x)
  \end{equation}
  for $\mu$-almost all $x \in X$ and
  \begin{equation} \label{eq:dim1_dimq}
    \lim_{q \downarrow 1} \dim_q(\mu,x) \le \ldim_1(\mu,x) \le \udim_1(\mu,x) \le \lim_{q \uparrow 1} \dim_q(\mu,x)
  \end{equation}
  for every $x \in \spt(\mu)$.

  Furthermore, if the measure $\mu$ has the density point property, then
  \begin{equation} \label{eq:local_dim1}
    \ldimloc(\mu,x) \le \ldim_1(\mu,x) \le \udim_1(\mu,x) \le \udimloc(\mu,x)
  \end{equation}
  for $\mu$-almost all $x \in X$.
\end{theorem}

The claims \eqref{eq:local_dimq} and \eqref{eq:local_dim1} are proved in \cite[Theorem 3.1 and Theorem 3.11]{KaenmakiRajalaSuomala2012} and \eqref{eq:dim1_dimq} follows immediately from Proposition \ref{prop:dim1} below. It is worthwhile to notice that the density point property is not needed in the global version of \eqref{eq:local_dim1} whereas in the local case, it is a necessary assumption; see \cite[Remark 3.12 and Examples 5.7--5.8]{KaenmakiRajalaSuomala2012}.

\section{Entropy and $L^q$-dimensions using partitions}\label{sec:par}

In this section, we reformulate the main definitions using partitions of the space $X$ and show that these definitions are consistent with the ordinary definitions presented above.
Concerning the global $L^q$-spectrum on $X$ and global entropy dimensions on $X$, this is of course a known result; see for instance \cite{Cutler1990, Olsen1995, Pesin1997}. We already saw in the inequality \eqref{eq:local_dim1} of Theorem \ref{thm:pointwise_dimq} that local and global definitions do not necessarily have the same basic properties. In Proposition \ref{1proposition} and Example \ref{1example}, we will see that global entropy dimensions on $A$ can be defined via partitions only when $A$ is compact.

The use of the partitions is motivated by the fact that the original definition of the $L^q$-dimension using packings often causes technical problems if $q<0$. Moreover, the measures that we are interested in usually have some additional a priori structure for which the partition definition suits well. For instance, see Lemma \ref{partitionlemma}. We also use the partition definitions to relate the $L^q$- and entropy dimensions in Proposition \ref{prop:dim1}.

Let $1\le \Lambda<\infty$.
A countable partition $\QQ$ of $X$ is called a \emph{$(\delta,\Lambda)$-partition} (for $\delta>0$) if all the sets of $\QQ$ are Borel sets and for each $Q \in \QQ$ there exists a ball $B_Q$ so that $Q \subset \Lambda B_Q$ and the collection $\{ B_Q : Q \in \QQ \}$ is a $\delta$-packing. The choice of $\Lambda$ is usually not important, and thus we simply talk about $\delta$-partitions and assume that $\Lambda$ has been silently fixed. Usually we consider $\delta_n$-partitions for a sequence of $\delta_n$ and in this case we assume that $\Lambda$ is the same for all $\delta_n$.

Let $(\delta_n)_{n \in \N}$ be a decreasing sequence of positive real numbers so that there is $0<c<1$ for which
\begin{equation}\label{eq:c}
  \delta_n<c^n
\end{equation}
for all $n$ and
\begin{equation}\label{log1}
\log\delta_n/\log\delta_{n+1}\longrightarrow 1
\end{equation}
as $n\rightarrow\infty$. For each $n\in\N$
we fix a $\delta_n$-partition $\QQ_n$.
If $x \in X$, then we
denote the unique element of $\QQ_n$ containing $x$ by $Q_n(x)$.
Furthermore, if $A \subset X$, then we set $\QQ_n(A) =\{ Q \in \QQ_n : A \cap
Q \ne \emptyset \}$ for all $n \in \N$.

Perhaps the most classical example of a $\delta$-partition is the dyadic cubes of the Euclidean space. We remark that in doubling metric spaces, it is possible to define similar kind of nested partitions sharing most of the good properties of dyadic cubes; see \cite{KaenmakiRajalaSuomala2012a} and references therein. But often in applications, the nested structure is inconvenient to work with. Since the $\delta_n$-partitions do not have to be nested, they are slightly more flexible than such generalised nested cubes.

Throughout this section, we assume that for each $n \in \N$ we have a fixed $\delta_n$-partition $\QQ_n$, where $(\delta_n)_{n \in \N}$ is a decreasing sequence satisfying \eqref{eq:c} and \eqref{log1}.

\subsection{Local dimensions via partitions}
We include a proof of the following folklore result, Proposition \ref{thm:localdim},
since we have not been able to track a complete proof in the literature
(see e.g.\ \cite[Lemma
2.3]{Cutler1990} and \cite[Theorem 15.3]{Pesin1997}).

To simplify the notation, we set
\begin{align*}
  \uld(\mu,x) &= \limsup_{n\to\infty} \log\mu(Q_n(x))/\log \delta_n, \\
  \lld(\mu,x) &= \liminf_{n\to\infty} \log\mu(Q_n(x))/\log \delta_n
\end{align*}
for all measures $\mu$ on $X$ and $x \in X$. A
priori, the definitions of $\uld(\mu,x)$ and $\lld(\mu,x)$ depend on
the choice of the partition, but Proposition \ref{thm:localdim} implies
that almost everywhere these quantities equal the local
dimensions and hence, the choice of the partition does not play any
role.

\begin{proposition}\label{thm:localdim}
If $\mu$ is a measure on a doubling metric space $X$, then
\begin{align*} 
  \udimloc(\mu,x) &= \uld(\mu,x), \\
  \ldimloc(\mu,x) &= \lld(\mu,x)
\end{align*}
for $\mu$-almost all $x\in X$.
\end{proposition}

\begin{proof}
The inequalities
  $\udimloc(\mu,x) \le \uld(\mu,x)$,
  $\ldimloc(\mu,x) \le \lld(\mu,x)$
are seen to hold for all $x\in X$ by using \eqref{log1} and the fact $Q_n(x)\subset B(x,(\Lambda+1)\delta_n)$ for all $x\in X$ and $n\in\N$.

To prove the estimates in the other direction,
fix a bounded set $A \subset X$, $0<t<s<\infty$ and define
\[A_n(t,s)=\{x\in A : \mu(Q_n(x))<\delta_{n}^{s}\text{ and }\mu(B(x,\delta_n))>\delta_{n}^t\}.\]
Now the set
\[\{x\in A : \udimloc(\mu,x)<\uld(\mu,x)\text{ or }\ldimloc(\mu,x)<\lld(\mu,x)\}\]
is contained in
\[\bigcup_{0<t<s<\infty}\bigcap_{k\in\N}\bigcup_{n=k}^\infty A_n(t,s)\,\]
where the union is over countably many (e.g.\ rational) $t$ and $s$. Thus,
by the Borel-Cantelli Lemma, it suffices to show that $\sum_{n\in\N}\mu(A_n(t,s))<\infty$ for any choice of $t$ and $s$.
To verify this, let $n\in\N$ and consider $x\in A_n(t,s)$. Since only at most $C=C(N,\Lambda)$ of the sets $Q\in\QQ_n$ meet $B(x,\delta_n)$ (cf.\ Lemma \ref{thm:covering_thm}\eqref{covering3}), we have the estimate
$\mu(A_n(t,s) \cap B(x,\delta_n))\le C \delta_n^{s}=C\delta_n^{s-t}\delta_{n}^t\le C\delta_n^{s-t}\mu(B(x,\delta_n))$. Using Lemma \ref{thm:covering_thm}\eqref{covering5}, we may cover $A_n(t,s)$ by a union of at most $M=M(N)$ $\delta_n$-packings of $A_n(t,s)$ and thus
\[\mu(A_n(t,s))\le CM \delta_n^{s-t}\mu(B),\]
where $B$ is a ball centered at $A$ with radius $\diam(A)+1$.
By \eqref{eq:c}, the sum $\sum_{n\in\N}\delta_{n}^{s-t}$ converges and the claim holds for $\mu$-almost all $x\in A$. As this is true for any bounded $A\subset X$, this finishes the proof.
\end{proof}

\subsection{$L^q$-spectrum and entropy dimension via partitions} \label{sec:partitions}
The equivalence of different definitions of $L^q$-spectum has already been considered in the literature,
in particular in \cite{Olsen1995}. Nevertheless, we present here a short proof of Proposition \ref{qproposition} for the convenience
of the reader. After that we study the more subtle case of entropy dimension, where the results, to our knowledge, are new.

The following proposition shows that both the local and global $L^q$-spectrum and $L^q$-dimension can equivalently be defined by using partitions.
Later we will show that this is also the case for the local entropy dimension, see Proposition \ref{1proposition}. For the global entropy dimension the situation is slightly more complicated.

\begin{proposition}\label{qproposition}
  If $\mu$ is a measure on a doubling metric space $X$, $A \subset X$ is bounded with $\mu(A)>0$ and $q\ge0$, then
    \begin{equation*}
      \tau_q(\mu,A) = \liminf_{n\rightarrow\infty}
      \frac{\log\sum_{Q \in \QQ_n(A)} \mu(Q)^q}{\log \delta_n}.
    \end{equation*}
\end{proposition}

\begin{proof}
Let $0<\delta<\delta_1$, and $n\in\N$ so that
$\delta_{n+1}\le\delta<\delta_n$.
Our first goal is to show that for a constant
$c_1=c_1(N,\Lambda,q)>0$, we have
  \begin{equation} \label{eq:S_first_goal}
    S_{q}(\mu,A,\delta) \le c_1\Bigl(\frac{\delta_n}{\delta}\Bigr)^s\sum_{Q \in \QQ_n(A)} \mu(Q)^q,
  \end{equation}
  where $s = s(N) >0$ is the constant given by Lemma \ref{thm:covering_thm}\eqref{covering3}.
  Recall that $N$ is the doubling constant of
  $X$ and $\Lambda$ is the fixed constant used in defining the
  partitions $\QQ_n$.
 To show \eqref{eq:S_first_goal}, we fix a $\delta$-packing $\BB$ of
 $A$ and let
  \begin{equation*}
    \CC_B = \{ Q \in \QQ_n(A) : Q \cap B \ne \emptyset \}
  \end{equation*}
  for all $B \in \BB$. Since $\CC_B$ is a cover for $B$, we have
  \begin{equation*}
    \mu(B)^q \le \biggl( \sum_{Q \in \CC_B} \mu(Q) \biggr)^q \le (\#\CC_B)^q \sum_{Q \in \CC_B} \mu(Q)^q,
  \end{equation*}
  where $\#\CC_B$ is the cardinality of $\CC_B$. Notice that all
  the sets of $\CC_B$ are contained in a ball of radius
  $(1+2\Lambda)\delta_n$ which, on the other hand, has a $\delta_n$-packing of cardinality $\#\CC_B$. Hence, Lemma \ref{thm:covering_thm}\eqref{covering3} implies that $\#\CC_B \le c_2=c_2(N,\Lambda)$ for all $B\in\BB$ and therefore
  \begin{equation*}
    \sum_{B \in \BB} \mu(B)^q \le c_2^q \sum_{B \in \BB} \sum_{Q \in \CC_B} \mu(Q)^q.
  \end{equation*}
  Furthermore, by Lemma \ref{thm:covering_thm}\eqref{covering3} there exists a constant $c_3 = c_3(N,\Lambda)>0$
  so that the cardinality of the set $\{ B \in \BB : Q \cap B \ne \emptyset \}$ is at most $c_3\left(\delta_n/\delta\right)^s$
  for all $Q \in \QQ_n$. Thus, \eqref{eq:S_first_goal} follows with $c_1=c_2^q c_3$.

To find an estimate in the other direction, choose for each $Q \in \QQ_n(A)$ a point $x_Q \in A \cap Q$ and a ball $B_Q$ so that $Q \subset \Lambda B_Q$ and the collection $\{ B_Q : Q \in \QQ_n(A) \}$ is a $\delta_n$-packing. Notice that $Q \subset B(x_Q,2\Lambda\delta_n) \subset 3\Lambda B_Q$ for all $Q \in \QQ_n(A)$. According to Lemma \ref{thm:covering_thm}\eqref{covering4} there exists $M = M(N,\Lambda) \in \N$ and $\QQ_1,\ldots,\QQ_M$ so that $\QQ_n(A) = \bigcup_{i=1}^M \QQ_i$ and $\{ 3\Lambda B_Q : Q \in \QQ_i \}$ is a $3\Lambda\delta_n$-packing for all $i \in \{ 1,\ldots,M \}$. Thus $\{ B(x_Q,2\Lambda\delta_n) : Q \in \QQ_i \}$ is a $2\Lambda\delta_n$-packing of $A$ for all $i \in \{ 1,\ldots,M \}$. Since $\bigcup_{Q \in \QQ_n(A)} Q \subset \bigcup_{i=1}^M \bigcup_{Q \in \QQ_i} B(x_Q,2\Lambda\delta_n)$, we may choose $i \in \{ 1,\ldots,M \}$ so that
\begin{equation} \label{eq:S_second_goal}
  \sum_{Q \in \QQ_n(A)} \mu(Q)^q \le M^{-1}\sum_{Q \in \QQ_i} \mu(B(x_Q,2\Lambda\delta_n))^q \le M^{-1}S_q(\mu,A,2\Lambda\delta_n).
\end{equation}
The proof now follows by combining \eqref{eq:S_first_goal} and \eqref{eq:S_second_goal} and taking logarithms and limits.
\end{proof}

Global entropy dimensions can be defined via partitions if $A$ is compact. Before showing this, we exhibit a small technical lemma.

\begin{lemma}\label{lma:smallentropy}
  Suppose $\mu$ is a measure on a doubling metric space $X$ and $A \subset X$ is bounded. Let $s>0$ and $c>0$ be as in Lemma \ref{thm:covering_thm}\eqref{covering3}. Then
  \[
    \int_A \log\mu(B(y,\delta))\,d\mu(y) \ge - \tfrac1e -\mu(A)\biggl( \log c + s\log\frac{4\diam(A)}{\delta} \biggr)
  \]
  for all $\delta > 0$.
\end{lemma}

\begin{proof}
  Let $\BB'$ be a maximal $\delta/4$-packing of $A$ and let $\BB = \tfrac14 \BB' = \{ B_1,B_2,\ldots, B_k\}$. Define $Q_1 = 8B_1 \setminus \bigcup_{B \in \BB \setminus \{ B_1 \}} B$ and
  \begin{equation*}
    Q_{n+1} = \biggl( 8B_{n+1} \setminus \bigcup_{B \in \BB \setminus \{ B_{n+1} \}} B \biggr) \setminus \bigcup_{i=1}^n Q_i
  \end{equation*}
  for all $1\le n\le k-1$. Then $\QQ = \{ Q_1 \cap A, Q_2  \cap A, \ldots, Q_k\cap A \}$ is a $(\delta/16,8)$-partition of $A$ in the relative metric.

  If the unique $Q \in \QQ$ containing $y \in A$ is denoted by $Q(y)$, then we have $Q(y) \subset B(y, \delta)$ for all $y \in A$. By Theorem \ref{thm:covering_thm}\eqref{covering3} there exist constants $s>0$ and $c>0$ depending only on the doubling constant of $A$ so that $\#\QQ \le c(4\diam(A)/\delta)^{s}$. Now Jensen's inequality gives
 \begin{align*}
  \int_A \log\mu(B(y,\delta))\,d\mu(y) & \ge \int_A \log\mu(Q(y))\,d\mu(y) = \sum_{Q \in \QQ}\mu(Q)\log\mu(Q) \\
  &\ge \mu(A)\log\frac{\mu(A)}{\#\QQ} \ge \mu(A)\log \mu(A) - \mu(A)\log \frac{c(4\diam(A))^s}{\delta^s}
 \end{align*}
  and the claim follows.
\end{proof}

\begin{proposition}\label{1proposition}
  If $\mu$ is a measure on a doubling metric space $X$ and $A \subset X$ is compact with $\mu(A)>0$, then
    \begin{align*}
      \udim_1(\mu,A) &=
      \limsup_{n\rightarrow\infty} \frac{\sum_{Q \in \QQ_n(A)}
      \mu(Q)\log\mu(Q)}{\sum_{Q \in \QQ_n(A)}
      \mu(Q) \log \delta_n}, \\
      \ldim_1(\mu,A) &=
      \liminf_{n\rightarrow\infty} \frac{\sum_{Q \in \QQ_n(A)}
      \mu(Q)\log\mu(Q)}{\sum_{Q \in \QQ_n(A)}
      \mu(Q)\log\delta_n}.
    \end{align*}
\end{proposition}

\begin{proof}
Choose for each $Q \in \QQ_n(A)$ a ball $B_Q$ such that $Q \subset \Lambda B_Q$ and $\{ B_Q : Q \in \QQ_n(A) \}$ is a $\delta_n$-packing. If $Q \in \QQ_n(A)$, then for every $y \in Q$ we have
\begin{equation*}
  Q \subset B(y,2\Lambda\delta_n) \subset 3\Lambda B_Q \subset
  \bigcup_{Q' \in \CC_Q} Q',
\end{equation*}
where $\CC_Q = \{ Q' \in \QQ_n(A) : Q' \cap 3\Lambda B_Q \ne \emptyset \}$. Thus, letting $A_n = \bigcup_{Q \in \QQ_n(A)} Q$, we get
\begin{equation*}
\begin{split}
  \sum_{Q \in \QQ_n(A)} \mu(Q)\log\mu(Q) &\le \sum_{Q \in \QQ_n(A)} \int_{Q} \log\mu(B(y,2\Lambda\delta_n))\,d\mu(y) \\
  &= \int_{A_n} \log\mu(B(y,2\Lambda\delta_n))\,d\mu(y) \le \sum_{Q \in \QQ_n(A)} \mu(Q)\log\sum_{Q' \in \CC_Q} \mu(Q').
\end{split}
\end{equation*}
Moreover, since each $Q' \in \QQ_n(A)$ is contained in at most $c_4(3\Lambda)^{s}$ collections $\CC_Q$ by Lemma \ref{thm:covering_thm}(3), where $c_4=c_4(N)<\infty$, we have
 \begin{align*}
  \sum_{Q \in \QQ_n(A)} \mu(Q) \log&\sum_{Q' \in \CC_Q} \mu(Q')
   -\sum_{Q \in \QQ_n(A)} \mu(Q)\log\mu(Q) \\
  &= \sum_{Q \in \QQ_n(A)} \mu(Q)\log\biggl( 1 + \frac{\sum_{Q' \in \CC_Q \setminus \{ Q \}} \mu(Q')}{\mu(Q)} \biggr) \\
  &\le \sum_{Q \in \QQ_n(A)} \sum_{Q' \in \CC_Q \setminus \{ Q \}} \mu(Q') \le c_4(3\Lambda)^s\mu(B_0),
 \end{align*}
where $B_0$ is a ball centered at $A$ with radius $\diam(A) + 2\Lambda\delta_n$. Putting these estimates together, we get
 \begin{equation}\label{eq:comb}
 \begin{split}
   \sum_{Q \in \QQ_{n}(A)} \mu(Q)\log\mu(Q) &\le \int_{A_n} \log\mu(B(y,\delta))\,d\mu(y) \\ &\le \sum_{Q \in \QQ_{n-1}(A)} \mu(Q)\log\mu(Q) + c_4(3\Lambda)^s\mu(B_0)
 \end{split}
 \end{equation}
for all $2\Lambda\delta_n \le \delta \le 2\Lambda\delta_{n-1}$.

Since $A$ is compact we have $\lim_{n \to \infty} \mu(A_n \setminus A) = 0$ and therefore, by Lemma \ref{lma:smallentropy},
\[
  \lim_{n \to \infty}\frac{1}{\log \delta_n}\int_{A_n \setminus A} \log\mu(B(y,2\Lambda\delta_n))\,d\mu(y) = 0.
\]
From this, \eqref{log1} and \eqref{eq:comb} the claim follows easily.
\end{proof}

\begin{example} \label{1example}
  In this example, we show that the claim in Proposition \ref{1proposition} does not hold for non-compact sets. Equip $X = [0,1]$ with the Euclidean metric and let $\QQ_n$ be the partition of $X$ to the dyadic intervals of length $2^{-n}$. Let $A = \Q \cap [0,1] = \{ q_1,q_2,\ldots \}$ and $\nu = \sum_{i=1}^\infty 2^{-i} \delta_{q_i}$, where $\delta_x$ denotes the Dirac unit mass located at $x$. Finally, set $\mu = \LL^1|_{[0,1]} + \nu$, where $\LL^1$ denotes the Lebesgue measure.

  Since $\sum_{Q \in \QQ_n} \mu(Q)\log\mu(Q) \le \log 2^{-n}$ for all $n$ large enough,
  we have
  \begin{equation*}
    \liminf_{n\rightarrow\infty} \frac{\sum_{Q \in \QQ_n(A)}
      \mu(Q)\log\mu(Q)}{\sum_{Q \in \QQ_n(A)}
      \mu(Q)\log 2^{-n}} = \liminf_{n \to \infty} \frac{\sum_{Q \in \QQ_n} \mu(Q) \log\mu(Q)}{2\log 2^{-n}} \ge \tfrac12.
  \end{equation*}
  Let $\eps>0$, choose $k \in \N$ so that $\sum_{i=k+1}^\infty 2^{-i} < \eps$, and define $A' = \{ q_1,\ldots,q_k \}$. According to Lemma \ref{lma:smallentropy}, there exists $c > 0$ so that
  \begin{align*}
    \int_{A'} \log\mu(B(y,\delta))\,d\mu(y) &\ge -\tfrac{1}{e} - \nu(A')k \log 2, \\
    \int_{A \setminus A'} \log\mu(B(y,\delta))\,d\mu(y) &\ge -\tfrac{1}{e} - \nu(A \setminus A')\bigl( \log c + \log(4/\delta) \bigr)
  \end{align*}
  for all $\delta>0$ small enough.
  Since $\nu(A \setminus A')<\eps$, we get
  \begin{equation*}
    \udim_1(\mu,A) = \limsup_{\delta \downarrow 0} \int_A \frac{\log\mu(B(y,\delta))}{\log\delta}  \le \limsup_{\delta \downarrow 0} \frac{-\tfrac{2}{e} - \nu(A')k\log 2 - \eps\log(4c) + \eps\log\delta}{\log\delta} = \eps.
  \end{equation*}
  Thus $\udim_1(\mu,A) = 0$.
\end{example}

\begin{remark}
In view of the definitions of $\dim_q$, it is natural to ask if the entropy dimensions could also be defined in terms of maximal packings. However, simple examples such as
$\mu = \LL^1|_{[0,1]} + \delta_1$ on $[0,1]$ show that this is usually not possible.
\end{remark}

To finish this section, we show that the definition of the entropy dimension as $\dim_1$ is consistent with the monotonicity of the $L^q$-dimensions.
The proof is standard and it is presented for the convenience of the reader.

\begin{proposition} \label{prop:dim1}
  If $\mu$ is a measure on a doubling metric space $X$ and $A\subset X$ compact with $\mu(A) > 0$, then
  \begin{equation*}
    \lim_{q \downarrow 1} \dim_q(\mu,A) \le \ldim_1(\mu,A) \le \udim_1(\mu,A) \le \lim_{q \uparrow 1} \dim_q(\mu,A).
  \end{equation*}
\end{proposition}

\begin{proof}
The existence of the limits follows from \cite[Proposition 2.7]{KaenmakiRajalaSuomala2012}. Thus, the claims follow if we can show that
\begin{equation}\label{goal}
\tau_q(\mu,A)/(q-1) \ge \udim_1(\mu,A) \ge \ldim_1(\mu,A) \ge \tau_p(\mu,A)/(p-1),
\end{equation}
where $0<q<1<p$.
Define $h_n(q) = \log \sum_{Q \in \QQ_n(A)} \mu(Q)^q$ for all $q\geq0$.
A simple application of H\"older's inequality shows that $h_n$ is
convex. As $\QQ_n(A)$ has only a finite number of elements, $h_n$
is differentiable with $h_n'(1) = \bigl( \sum_{Q \in \QQ_n(A)} \mu(Q) \bigr)^{-1} \sum_{Q \in \QQ_n(A)} \mu(Q)\log\mu(Q)$. Thus
  \begin{equation*}
    \frac{h_n(q)-h_n(1)}{q-1} \le h_n'(1) \le \frac{h_n(p)-h_n(1)}{p-1}.
  \end{equation*}
Using these estimates and the fact that $h_n(1)=\log\sum_{Q \in \QQ_n(A)}\mu(Q)$ does not depend on $n$, we calculate
\begin{align*}
\frac{1}{q-1}&\liminf_{n\to\infty}
\frac{\log\sum_{Q\in\QQ_n(A)}\mu(Q)^{q}}{\log \delta_n}=\limsup_{n\to\infty}
  \frac{h_n(q)-h_n(1)}{(q-1)\log \delta_n}\\
&\geq\limsup_{n\to\infty}\frac{\sum_{Q \in \QQ_n(A)}
\mu(Q)\log\mu(Q)}{\sum_{Q \in \QQ_n(A)}\mu(Q)\log \delta_n}
\geq\liminf_{n\to\infty}\frac{\sum_{Q \in \QQ_n(A)}
\mu(Q)\log\mu(Q)}{\sum_{Q \in \QQ_n(A)}\mu(Q)\log \delta_n}\\
&\geq\liminf_{n\to\infty}
  \frac{h_n(p)-h_n(1)}{(p-1)\log \delta_n}=\frac1{p-1}\liminf_{n\to\infty}
  \frac{\log\sum_{Q\in\QQ_n(A)}\mu(Q)^{p}}{\log \delta_n}.
\end{align*}
The desired estimate \eqref{goal} now follows from Propositions \ref{qproposition} and \ref{1proposition}.
\end{proof}

\section{Local dimension and multifractal analysis for Moran measures}\label{sec:results}

We now turn towards our final goal to study multifractality of measures in metric spaces.
We first introduce a class of Moran constructions
in a complete doubling metric space $X$ and show how
Theorem \ref{thm:pointwise_dimq} can be applied to calculate the local
dimensions for a large class of measures defined on these Moran
fractals. Then, under certain additional assumptions, we turn to study the multifractal spectrum of these measures. Our main aim is to show that using the technique
introduced in \cite{KaenmakiRajalaSuomala2012}, we can push the standard methods
used to calculate the local dimensions for self-similar measures on
Euclidean spaces (see \cite{CawleyMauldin1992, Riedi1995, Falconer1997})
to obtain analogous results in doubling
metric spaces with very mild regularity assumptions, see Remark \ref{rem:moran}.

\subsection{Moran constructions and measures}

Let $m\in\N$,
$\Sigma = \{ 1,\ldots,m \}^\N$, $\Sigma_n = \{
1,\ldots,m \}^n$ for all $n \in \N$, and $\Sigma_* = \{\varnothing\}\cup\bigcup_{n \in
  \N} \Sigma_n$. If $n \in \N$ and $\iii \in \Sigma \cup
\bigcup_{j=n}^\infty \Sigma_j$, then we let $\iii|_n =
(i_1,\ldots,i_n)$ (and $\iii|_0=\varnothing$). The concatenation of two words $\iii \in \Sigma_*$
and $\jjj \in \Sigma \cup \Sigma_*$ is denoted by $\iii\jjj$. We also
set $\iii^- = \iii|_{n-1}$ for $\iii \in \Sigma_n$ and $n \in
\N$. By $|\iii|$, we  denote the length of a word $\iii\in\Sigma_*$.
We assume that
$\{ E_\iii : \iii \in \Sigma_* \}$ is a collection of compact subsets of
$X$ that satisfy the following conditions for some constants $0<C_0,C_1<\infty$:
\begin{enumerate}
  \item[(M1)] $E_\iii \subset E_{\iii^-}$ for all $\varnothing\neq\iii \in \Sigma_*$.\label{M1}
  \item[(M2)] $E_{\iii i}\cap E_{\iii j}=\emptyset$ if\label{M2}
    $\iii\in\Sigma_*$ and $i\neq j$.
  \item[(M3)] For each $\iii\in\Sigma_*$, there is $x\in E_\iii$ such
    that $B\left(x,C_0 \diam(E_\iii)\right)\subset E_\iii$.\label{M3}
  \item[(M4)] $\diam(E_{\iii|_n})\rightarrow 0$ as
    $n\rightarrow\infty$, for each $\iii\in\Sigma$.
  \item[(M5)] $\diam(E_{\iii^-})/\diam(E_{\iii})\leq C_1<\infty$ for all $\varnothing\neq\iii\in\Sigma_*$.\label{M5}
\end{enumerate}
We define the limit set of the construction as
$E=\bigcap_{n\in\N}\bigcup_{\iii\in\Sigma_n}E_\iii$ and given
$\iii\in\Sigma$, denote by $x_\iii$ the point obtained as
$\{x_\iii\}=\bigcap_{n\in\N}E_{\iii|_n}$. If $\iii\in\Sigma_*$, we denote $x_\iii:=x_{\iii000\cdots}$.
We assume that for each
  $i\in\{1,\ldots,m\}$ there is a continuous function $r_i\colon
  E\to(0,1)$. Given $x\in E$, and $\iii\in\Sigma_n$, we let
  $r_\iii(x)=\prod_{k=1}^{n}r_{i_k}(x)$. Moreover, we assume that
\begin{enumerate}
  \item[(M6)]
    $\lim_{n\rightarrow\infty}\log \diam(E_{\iii|_n})/\log r_{\iii|_n}(x_\iii)=1$
    uniformly for all $\iii\in\Sigma$.
\end{enumerate}

The following further conditions on the collection $\{E_\iii : \iii\in\Sigma_*\}$ are needed only in subsection \ref{lmf}. For $n\in\N$, denote
$\mathcal{E}_n=\{E_\iii\,:\,\diam(E_\iii)\le C_1/(C_0 2^{n})<\diam(E_{\iii^-})\}$.
\begin{enumerate}
\item[(M7)] There is $c>0$ so that for each $Q\in\mathcal{E}_n$, there
  is $x\in E$ with $B(x,c2^{-n})\subset Q$.
\item[(M8)] $\lim_{r\downarrow 0}\log r/\log\bigl(\diam(E_{\iii|_{n(\iii,r)}})\bigr)=1$ for all $\iii\in\Sigma$, where $n(\iii,r)=\max\{n\in\N : B(x_\iii,r)\cap E \subset E_{\iii|_n}\}$.
\end{enumerate}

Our next lemma shows how we can obtain a
$\delta$-partition of $X$ from the elements of $\{E_\iii\}$ which are
roughly of size $\delta$.

\begin{lemma}\label{partitionlemma}
If $\{ E_\iii : \iii \in \Sigma_* \}$ is a collection of compact sets satisfying the conditions (M1)--(M5), then for every $n \in \N$ there is a $(2^{-n})$-partition $\mathcal{Q}_n$ of $X$ such that each $E_\iii\in\mathcal{E}_n$ is a subset of some $Q\in\mathcal{Q}_n$ and all elements of $\mathcal{Q}_n$ contain at most one element of $\mathcal{E}_n$.
\end{lemma}

\begin{proof}
Consider a maximal collection $\mathcal{B}_n$ of disjoint balls of radius
$2^{-n}$ contained in $X\setminus\bigcup\mathcal{E}_n$. Define $\mathcal{A}_n=\mathcal{E}_n\cup\mathcal{B}_n=\{A_1,A_2,\ldots\}$. For each
$x\in X$,
we let
$i_x=\min\{j\in\N\,:\,\dist(x,A_j)=\min_{A\in\mathcal{A}_n}\dist(x,A)\}$
and
set $Q_{A_i}=\{x\in X\,:\, i_x=i\}$ for all $i\in\N$.
It is
then easy to see that $\mathcal{Q}_n=\{Q_A\,:A\in\mathcal{A}_n\}$ is the desired
$(2^{-n})$-packing. Observe that each $Q\in\mathcal{Q}_n$ is a Borel set
since $\bigcup_{i=1}^k Q_{A_i}$ is closed for all $k$. Moreover, the constant $\Lambda$ of this
partition depends only on the constants $C_0$ and $C_1$ as one may choose $\Lambda=C_0C_1+1$.
\end{proof}

Let $\mu$ be a probability measure on $X$ with $\spt(\mu)=E$. Then for each
$\iii\in\Sigma_*$, $\mu$ induces a probability
vector $p_\iii=(p_\iii^{1},\ldots,p_\iii^{m})$ with $p_\iii^{i}>0$ for $i\in\{1,\ldots,m\}$ such that $\mu(E_{\iii
i})=p_\iii^{i}\mu(E_\iii)$ for $i\in\{1,\ldots,m\}$. Given $\iii\in\Sigma_n$, we denote
$\mu_\iii:=\mu(E_\iii)=\prod_{j=1}^{n}p_{\iii|_{j-1}}^{i_j}$. In the next theorem, we assume that the weights $p_\iii$ are controlled in terms of continuous probability functions $p(x)=\bigl( p_1(x),\ldots,p_m(x) \bigr)$. More precisely, we assume that
 for each $i\in\{1,\ldots,m\}$, the function $p_i\colon
E\rightarrow(0,1)$ is continuous with $\sum_{i=1}^m p_i(x)=1$ for
all $x\in E$.
Similarly to $r_i$, we define $P_\iii(x)=\prod_{k=1}^{n}p_{i_k}(x)$ when $\iii\in\Sigma_n$.

\subsection{Local $L^q$-spectrum for Moran measures}

\begin{theorem}\label{thm:localdimappl}
Let $\{E_\iii : \iii\in\Sigma_*\}$ be a collection of compact sets that satisfy the conditions (M1)--(M6). Suppose that $\mu$ is a probability measure on $E$ and let $p_\iii$ and $p$ be as above. If
$p_{\iii|_n}\rightarrow p(x_\iii)$ as $n\rightarrow\infty$ uniformly
for all $\iii\in\Sigma$, then, for all $x\in E$ and all $q\ge0$,
$\tau_q(\mu,x)$ is the unique $\tau\in\R$ that satisfies
\begin{equation}\label{tauformula}
\sum_{i=1}^m p_i(x)^qr_i(x)^{-\tau}=1.
\end{equation}
Moreover,
\begin{equation}\label{dimformula}
  \dim_1(\mu,x) = \dimloc(\mu,x) = \frac{\sum_{i=1}^m p_i(x)\log
    p_i(x)}{\sum_{i=1}^m p_i(x)\log r_i(x)}
\end{equation}
for $\mu$-almost all $x \in E$.
\end{theorem}

\begin{proof}
We prove the claim \eqref{tauformula}. The identities \eqref{dimformula}
then follow from \eqref{tauformula} by implicit differentiation together with
Theorem \ref{thm:pointwise_dimq}.

For each $n\in\N$, let $\mathcal{Q}_n$ be as in
Lemma \ref{partitionlemma}. Given $\varnothing\neq\iii\in\Sigma_*$, we denote by
$Q_\iii$ the unique element of $\bigcup_{n\in\N}\mathcal{Q}_n$ that
contains $E_\iii$ and does not contain $E_{\iii^-}$ (we assume without
loss of generality that $\mathcal{E}_1 = \{ E_\varnothing \}$ so that this makes
sense for all $n$).
Let us fix $q\ge 0$, $x\in E$ and let $\iii\in\Sigma$ so that $x=x_\iii$. Let
$\tau$ be as in \eqref{tauformula}. We first prove that
$\tau_q(\mu,x)\ge\tau$. Let $0<c<1$.
Since
$p_{\iii|_n}\rightarrow p(x_\iii)$ uniformly and $y\mapsto p(y)$ is
continuous, we may choose $n_0$ so large that
$p_{\jjj}^i>cp_i(x)$ whenever $i\in\{1,\ldots,m\}$, $\jjj\in\Sigma_*$,
 and $E_\jjj\subset E_{\iii|_{n_0}}$. Making $n_0$ even larger if
necessary, we may also assume that
\begin{equation}\label{tahhti}
c r_i(y)\le r_i(x)\le
r_i(y)/c
\end{equation}
for all $y\in E_{\iii|_{n_0}}$ and all $i \in \{1,\ldots,m\}$.

Now, for all $r>0$, we
choose $N_0\ge n_0$ so that $Q_{\jjj}\subset B(x,r)$ whenever
$\jjj\in\Sigma_*$,
and $E_\jjj\subset E_{\iii|_{N_0}}$. Given
$n\ge N_0$, let
$Z_{n}=\{\jjj\in\Sigma_* : Q_\jjj\in\mathcal{Q}_n\text{ and
}E_\jjj\subset E_{\iii|_{N_0}}\}$.
Let $\varepsilon_n=\min_{\jjj\in
  Z_{n}}\bigl(\diam(E_\jjj)/r_\jjj(x)\bigr)^{-\tau}$. Now, denoting
$c_0=C_{1}^{-|\tau|}\mu_{\iii|_{N_0}}^q$,
we get an estimate
\begin{equation}\label{alku}
\begin{split}
2^{n\tau}\sum_{\jjj\in Z_{n}}\mu(Q_\jjj)^q&\ge C_{1}^{-|\tau|}\sum_{\jjj\in
  Z_{n}}\mu_{\jjj}^q \diam(E_{\jjj})^{-\tau}\ge c_0\varepsilon_n
\sum_{\jjj\in Z_{n}}c^{q|\jjj|}P_{\jjj}(x)^q r_{\jjj}(x)^{-\tau}.
\end{split}
\end{equation}
For each $\jjj\in Z_{n}$, pick $y\in E_\jjj$.
Using (M6),
we may assume that $\log r_\jjj(y)\ge
2 \log \diam(E_\jjj) $ by making $N_0$ larger if necessary. Letting
$r_{\max}=\max\bigl\{r_i(y) : y\in E \text{ and } i \in \{ 1,\ldots,m\} \bigr\}$, we have
\begin{equation*}
\log r_{\max}^{|\jjj|}\ge \log r_\jjj(y)\ge 2\log
\diam(E_\jjj)\ge -2n\log 2,
\end{equation*}
and consequently,
\begin{equation}\label{kaks}
|\jjj|\le \frac{2\log \diam(E_{\jjj})}{\log r_{\max}} \le\frac{-2n \log 2}{\log r_{\max}}\le n C_2,
\end{equation}
for a constant $C_2<\infty$ independent of $n$.
On the other hand,
\begin{equation}\label{kol}
\sum_{\jjj\in
  Z_{n}}P_\jjj(x)^q r_{\jjj}(x)^{-\tau}=P_{\iii|_{n_0}}(x)^q
r_{\iii|_{n_0}}(x)^{-\tau}=: C_3
\end{equation}
by iterative use of \eqref{tauformula}.
Putting \eqref{alku}--\eqref{kol} together, we get
\begin{equation}\label{nel}
\log\sum_{\jjj\in Z_{n}}\mu(Q_\jjj)^q\ge\log(2^{-n\tau} c_0 \eps_n c^{qnC_2} C_3)
=\log 2^{-n\tau}+\log\varepsilon_n +q n C_2\log c+\log (c_0C_3).
\end{equation}
To estimate $\log \varepsilon_n$, we choose $\jjj\in Z_n$ such
that $\varepsilon_n=\bigl(\diam(E_\jjj)/r_\jjj(x)\bigr)^{-\tau}$. Then
\begin{equation}\label{kuus}
\log\varepsilon_n=-\tau\log \diam(E_\jjj)\bigl(1-\log r_\jjj(x)/\log \diam(E_\jjj)\bigr).
\end{equation}
Moreover,
$\log r_{\jjj}(y)+|\jjj|\log c\le\log r_{\jjj}(x)\le\log r_{\jjj}(y)-|\jjj|\log c$
for all $y\in E_\jjj$ by \eqref{tahhti}.
Using \eqref{kaks}, this gives
\begin{equation}\label{viis}
\frac{\log r_\jjj(y)}{\log \diam(E_\jjj)}+C_4\log c\le\frac{\log
  r_\jjj(x)}{\log \diam(E_\jjj)}\le\frac{\log r_\jjj(y)}{\log \diam(E_\jjj)}-C_4\log c
\end{equation}
for some constant $0<C_4<\infty$.

Using \eqref{nel}, \eqref{kuus}, \eqref{viis}, and (M6), we finally get
\begin{align*}
\liminf_{n\rightarrow\infty}\frac{\log \sum_{Q \in \QQ_n(B(x,r))} \mu(Q)^q}{\log 2^{-n}}&\le\liminf_{n\rightarrow\infty}\frac{\log\sum_{\jjj\in
    Z_{n}}\mu(Q_\jjj)^q}{\log 2^{-n}} \le\tau-(qC_2+|\tau|C_4)\log c/\log 2.
\end{align*}
As $c<1$ and $r>0$ can be chosen arbitrarily, we get, by recalling Proposition \ref{qproposition}, that $\tau_q(\mu,x)\le\tau$.

To prove that $\tau_q(\mu,x)\ge\tau$, we first fix $0<c<1$ and $r_0>0$ so
that $cp_\jjj^i<p_i(x)$ and $cr_i(x)<r_i(y)<\tfrac1c r_i(x)$ whenever $i\in\{1,\ldots,m\}$ and
$y\in E_\jjj\subset B(x,r_0)$. Then, if $0<r<r_0$, we may
find $n_0\in\N$ and finitely many elements
$E_\kkk\in\mathcal{Q}_{n_0}$, $E_\kkk\subset B(x,r_0)$
whose union covers $B(x,r)$. For each such
$E_\kkk$, and $n\ge n_0$, we put
$Z_{n,\kkk}=\{\jjj\in\Sigma_*\,:\,Q_\jjj\in\mathcal{Q}_n\text{ and
}E_\jjj\subset E_\kkk \}$. Putting $M_n=\max_{\jjj\in
  Z_{n,\kkk}}\bigl(\diam(E_\jjj)/r_\jjj(x)\bigr)^{-\tau}$, we may estimate as in
\eqref{alku} to obtain
\begin{equation*}
2^{n\tau}\sum_{\jjj\in Z_{n,\kkk}}\mu(Q_\jjj)^q\le C_5M_{n}
\sum_{\jjj\in Z_{n,\kkk}}c^{-q|\jjj|}P_\jjj(x)^qr_{\jjj}(x)^{-\tau}.
\end{equation*}
Calculating as above, this implies
\begin{equation*}
\liminf_{n\rightarrow\infty}\frac{\log \sum_{Q \in \QQ_n(B(x,r))}
  \mu(Q)^q}{\log 2^{-n}}\ge\tau+(qC_2+|\tau|C_4)\log c/\log 2.
\end{equation*}
Letting $r\downarrow 0$ and then $c\uparrow 1$, and using Proposition \ref{qproposition} gives $\tau_q(\mu,x)\ge\tau$.
\end{proof}

\begin{remark}\label{rem:moran}
(1) One can find Moran constructions that satisfy (M1)--(M6) on doubling metric spaces satisfying only
mild regularity assumptions on the space $X$. For instance, it suffices to assume that the space is uniformly perfect.
Different types of Moran constructions in metric spaces have been recently
studied in \cite{RajalaVilppolainen2009}.

(2) The result is interesting already in $\R^n$. We remark that
a self-similar measure on a self-similar set
satisfying the strong separation conditions is
a model case for Theorem \ref{thm:localdimappl} in the special case
when $p_i$ and $r_i$ are constant, see \cite{Falconer1997}.
However, as $p_i$ and $r_i$ are allowed to vary depending
on the point, Theorem \ref{thm:localdimappl} can be applied in
more general situations.

(3) One further difference to the self-similar situation
is that in Moran constructions the location of $E_\iii$ inside $E_{\iii^-}$ can be chosen quite freely,
whereas with similitude mappings the location of $E_\iii$ is strictly dictated by the maps.
Consequently, even in the simplest case where we would force a Moran construction in $\R^n$ to obey $\diam(E_\iii) = r^{|\iii|}$ for some $0<r<1$ and all $\iii \in \Sigma_*$, the limit set $E$ would not necessarily be bi-Lipschitz
 equivalent to a self-similar set.
\end{remark}

\subsection{Some multifractal analysis}\label{lmf}

To approach \eqref{eq:formalsim}, we have to deal with $\tau_q$ for negative values
of $q$ and for this we use the following lemma. Observe that we
cannot use Proposition \ref{qproposition} when $q<0$.

\begin{lemma}\label{taulemmaq<0}
Suppose that in the setting of Theorem \ref{thm:localdimappl} also
(M7) holds. Then, for all $x\in E$, $\tau_q(\mu,x)$ is determined by \eqref{tauformula} also when $q<0$.
\end{lemma}

\begin{proof}
Let $q<0$, $x\in E$ and let $\tau\in\R$ be the unique solution of \eqref{tauformula}.
With trivial modifications to the proof of Theorem \ref{thm:localdimappl}, we
see that
\begin{equation}\label{aa}
\tau=\lim_{t \downarrow 0} \liminf_{n\rightarrow\infty}
      \frac{\log\sum_{Q \in \QQ_{n,t}} \mu(Q)^q}{\log 2^{-n}}
\end{equation}
where $\mathcal{Q}_{n,t}=\{Q\in\QQ_n : Q\subset B(x,t)\text{ and
}Q\cap E\neq\emptyset\}$. (Observe that $\spt(\mu)=E$.)

In order to prove that $\tau=\tau_q(\mu,x)$, let $t>0$, $2^{-n}\le \delta<2^{-n+1}<t$ and $y\in E \cap B(x,t)$. Then there is $n_0\in\N$ depending only on the numbers $C_0$ and
$C_1$ so that $B(y,\delta)\supset Q_\iii$ for some
$Q_\iii\in\mathcal{Q}_{n+n_0,2t}$. Thus, for any $\delta$-packing
$\{B_i\}$ of $B(x,t)\cap\spt(\mu)$, we have
\begin{equation}\label{pee}
\sum_{i}\mu(B_i)^q\le\sum_{Q\in\QQ_{n+n_0,2t}}\mu(Q)^q.
\end{equation}

To get an estimate in the other direction, we fix $n$ and $t$ and use
the assumption (M7) to find for each $Q\in\mathcal{Q}_{n,t}$ a point
$y\in \spt(\mu)\cap B(x,t)$ such that for $B_Q=B(y,c 2^{-n})$, we have $B_Q\subset Q$.
Thus, for the $(c 2^{-n})$-packing $\{B_Q : Q\in\mathcal{Q}_{n,t}\}$, we have
\begin{equation}\label{see}
\sum_{Q\in\mathcal{Q}_{n,t}}\mu(B_Q)^q\ge\sum_{Q\in\mathcal{Q}_{n,t}}\mu(Q)^q.
\end{equation}
Combining \eqref{aa}--\eqref{see}, and taking logarithms, it follows
that $\tau_q(\mu,x)=\tau$.
\end{proof}

\begin{remark}
Inspecting the proofs of Theorem \ref{thm:localdimappl} and Lemma
\ref{taulemmaq<0}, we observe that in the setting of these results,
$\liminf_{n \to \infty}$ in the definition of $\tau_q(\mu,x)$ can actually be
replaced by $\lim_{n \to \infty}$.
\end{remark}

To complete the paper, we show how the local $L^q$-spectrum can be
used in the setting of Theorem \ref{thm:localdimappl}. We derive a
local multifractal formalism for the spectrum
\[
f_{\textrm{H}}(\alpha,x)=\lim_{r\downarrow
  0}\dimh\bigl(\{y\in B(x,r)\,:\,\dimloc(\mu,y)=\alpha\}\bigr)\]
for $x\in X$ and $\alpha\ge0$. The corresponding packing spectrum,
$f_{\textrm{p}}(\alpha,x)$ is defined by replacing $\dimh$ by $\dimp$ above.

Let $\alpha_{\min}(x)\le\alpha_{\max}(x)$ be the asymptotic derivatives of $q\mapsto\tau_q(\mu,x)$. Thus, $\alpha_{\min}(x) = \min\bigl\{ \log p_i(x)/\log r_i(x) : i \in \{ 1,\ldots,m \} \bigr\}$ and $\alpha_{\max}(x) = \max\bigl\{ \log p_i(x)/\log r_i(x) : i \in \{ 1,\ldots,m \} \bigr\}$. Then it is easy to check that
\begin{equation}\label{eq:min-max}
\bigcap_{r>0}\bigcup_{y\in B(x,r)\cap E}\left[\ldimloc(\mu,y),\udimloc(\mu,y)\right]\subset[\alpha_{\min}(x),\alpha_{\max}(x)]
\end{equation}
for
any $x\in E$.

\begin{theorem}\label{thm:localmultifractal}
Let $\{E_\iii : \iii\in\Sigma_*\}$ be a collection of compact sets that satisfy the conditions (M1)--(M8). If $p_{\iii|_n}\to p(x_\iii)$ uniformly for all $\iii\in\Sigma$, then
\[
  f_{\emph{H}}(\alpha,x) = f_{\emph{p}}(\alpha,x) = \inf_{q\in\R}\{\alpha q-\tau_q(\mu,x)\}
\]
for all $x\in E$ and $\alpha_{\min}(x)\le\alpha\le\alpha_{\max}(x)$.
\end{theorem}

\begin{corollary}\label{cor:spectrum}
Let $\{E_\iii : \iii\in\Sigma_*\}$ be a collection of compact sets that satisfy the conditions (M1)--(M8). Suppose that $r=(r_1,\ldots, r_m)$
and $p=(p_1,\ldots,p_m)$ are constant functions and $\mu$ is
a measure with $\spt(\mu)=E$ such that $p_{\iii|_n}\rightarrow p$
uniformly for all $\iii\in\Sigma$. If
$0\leq \alpha_{\min}\le\alpha_{\max}$ are the asymptotic
derivatives of the concave function $q\mapsto\tau_q(\mu)$,
then \eqref{eq:formalsim} holds for all $\alpha_{\min}\le\alpha\le\alpha_{\max}$.
\end{corollary}

\begin{remark}
For the inhomogeneous Bernoulli products on $[0,1]$, the claims of Theorem \ref{thm:localdimappl} and Corollary \ref{cor:spectrum} have been obtained independently by Batakis and Testud \cite[Corollary 1.3]{BatakisTestud2011}. In particular, our result shows that the result is true not only in higher dimensional Euclidean spaces, but also for all Moran constructions satisfying (M1)--(M8) in all doubling metric spaces.
\end{remark}

Theorem \ref{thm:localmultifractal} is derived from Lemma \ref{lemma:spectrum}
below.
The main idea is similar to that of the proof of \cite[Proposition
11.4]{Falconer1997}, but the definition of the auxiliary measures is perhaps more delicate in our setting.

Let $\sup_{x\in E}\alpha_{\min}(x)<\alpha<\inf_{x\in E}\alpha_{\max}(x)$.
In what follows, we use the following notation: Denote $f(\alpha,x)=\min_{q\in\R}\{\alpha q-\tau_q(\mu,x)\}$ for all $x\in E$
and by $q(x)=q_\alpha(x)$ the value of $q$ for which the minimum is attained. For $\iii\in\Sigma_*\cup\Sigma$ and $1\le i\le \iii$, we also set $r_\iii=r_\iii(x_\iii)$, $r_{\iii}^{i}=r_i(x_\iii)$, $q_\iii=q(x_\iii)$, and $\tau_\iii=\tau_{q_\iii}(\mu,x_\iii)$.

We make use of the following technical lemma.

\begin{lemma}\label{prop:tekninen}
Under the assumptions of Theorem \ref{thm:localmultifractal}, if
$\sup_{x\in E}\alpha_{\min}(x)<\alpha<\inf_{x\in E}\alpha_{\max}(x)$,
then
$r_{\iii|_n}\rightarrow r_\iii$, $q_{\iii|_n}\rightarrow q_\iii$ and $\tau_{\iii|_n}\rightarrow\tau_\iii$
 for all $\iii\in\Sigma$ unifromly as $n\rightarrow\infty$. Moreover, given $\eta>0$, there are $\delta>0$, $C<\infty$, and $\gamma<1$ such that
\begin{align}\label{eq:gammaexists}
&\sum_{\iii\in\Sigma_n}\mu_{\iii}^{q_\iii-\delta} r_{\iii}^{\delta(\alpha+\eta)-\tau_\iii}\le C\gamma^n, \\
&\sum_{\iii\in\Sigma_n}\mu_{\iii}^{q_\iii+\delta}
\label{eq:gammaexists2} r_{\iii}^{\delta(\eta-\alpha)-\tau_\iii}\le C\gamma^n,
\end{align}
for all $n\in\N$.
\end{lemma}

\begin{proof}
From our definition of the Moran construction, it follows that $r_{\iii|_n}\rightarrow r_\iii$ uniformly as $n\rightarrow\infty$.
We first recall that if $x=x_\iii$, then $\tau(q)=\tau_q(\mu,x)$ is given by the formula \eqref{tauformula} for all $q\in\R$ by Theorem \ref{thm:localdimappl} and Lemma \ref{taulemmaq<0}.
From this, the assumption $\sup_{x\in E}\alpha_{\min}(x)<\alpha<\inf_{x\in E}\alpha_{\max}(x)$ and the fact that all $p_i(x)$, $r_i(x)$ are bounded away from $0$ and $1$ by the compactness of $E$, it follows that there is a compact interval $I\subset\R$ such that $q_\iii\in I$ for all $\iii\in\Sigma$ (and thus also for all $\iii\in\Sigma_*$). Moreover,
using implicit differentation, it follows that there is $c>0$ such that for all $x\in E$, $q\in I$,
\begin{equation}\label{eq:d2}
\frac{d^2}{dq^2}\left(-\tau_q(\mu,x)\right)\ge c.
\end{equation}

Since $r_i(y)\rightarrow r_i(x)$ and $p_i(y)\rightarrow p_i(x)$ uniformly as $y\rightarrow x$, it follows that for all $q\in I$
\begin{equation}\label{eq:tau_unif}
\tau_q(\mu,x_{\iii|_n})\rightarrow \tau_q(\mu,x_\iii)
\end{equation}
for all $\iii\in\Sigma$, uniformly as $n\rightarrow\infty$. From this and the definition of $\tau_\iii$, it follows immediately that
\begin{equation}\label{eq:difference_converges}
  (\alpha q_{\iii|_n}-\tau_{\iii|_n}) \to (\alpha q_{\iii}-\tau_{\iii}),
\end{equation}
uniformly as $n\rightarrow\infty$.
Given $\varepsilon>0$, \eqref{eq:tau_unif} implies that there is $n_0\in\N$ independent of $\iii$, such that $|\tau_{\iii|_n}-\tau_{q_{\iii|_n}}(\mu,x_\iii)|<\varepsilon$ if $n\ge n_0$. Using \eqref{eq:d2}, we also have
\[\alpha q_{\iii|_n}-\tau_{q_{\iii|_n}}(\mu,x_{\iii})\ge\alpha q_{\iii}-\tau_{\iii}+\tfrac{c}{4}|q_{\iii|_n}-q_\iii|^2.\]
From these two estimates, we infer
\begin{align*}
\alpha q_{\iii|_n}-\tau_{\iii|_n}\ge \alpha q_{\iii}-\tau_{\iii}+\tfrac{c}{4}|q_{\iii|_n}-q_\iii|^2-\varepsilon,
\end{align*}
and combining with \eqref{eq:difference_converges}, we see that $q_{\iii|_n}\rightarrow q_\iii$ and $\tau_{\iii|_n}\rightarrow\tau_\iii$ uniformly for all $\iii\in\Sigma$.

To prove the estimate \eqref{eq:gammaexists}, fix $\eta>0$. We first observe by implicit differentation (see \cite[Lemma 11.3]{Falconer1997}), that there is $\delta>0$ and $\gamma<1$ such that for all $\iii$, we have
\begin{equation}
\sum_{i=1}^m (p_{\iii}^i)^{q_\iii-2\delta} (r_{\iii}^i)^{2\delta(\alpha+\eta)-\tau_\iii}\le\gamma.
\end{equation}
By the first part of the lemma, we may choose $n_0\in\N$, such that if $\iii=\iii_0\jjj\in\Sigma_*$, where $\iii_0\in\Sigma_{n_0}$, then $(p_{\iii}^ {i})^{q_\iii-\delta}(r_{\iii}^ {i})^{\delta(\alpha+\eta)-\tau_\iii}\le (p_{\iii_0}^{i})^{q_{\iii_0}-2\delta}(r_{\iii_0}^ {i})^{2\delta(\alpha+\eta)-\tau_{\iii_0}}$. This leads to
\begin{equation*}
\sum_{\iii=\iii_0 \jjj\in\Sigma_n}\mu_{\iii}^{q_\iii-\delta} r_{\iii}^{\delta(\alpha+\eta)-\tau_\iii}
\le C_{n_0}\left(\sum_{i=1}^m(p_{\iii_0}^{i})^{q_{\iii_0}-2\delta} (r_{\iii_0}^{i})^{2\delta(\alpha+\eta)-\tau_{\iii_0}}\right)^{n-n_0}.
\end{equation*}
Since there are only finitely many words $\iii_0\in\Sigma_{n_0}$, this yields \eqref{eq:gammaexists}.
The estimate \eqref{eq:gammaexists2} is proved in a similar manner.
\end{proof}

\begin{lemma}\label{lemma:spectrum}
In the setting of Theorem \ref{thm:localmultifractal}, let
$\underline{f}(\alpha)=\inf_{x\in E} f(\alpha,x)$ and $\overline{f}(\alpha)=\sup_{x\in E} f(\alpha,x)$.
Then
\begin{align}
\underline{f}(\alpha)\le\dimh(E_{\alpha})\le\dimp(E_{\alpha})\le
\overline{f}(\alpha),
\end{align}
for all $\sup_{x\in E}\alpha_{\min}(x)<\alpha<\inf_{x\in E}\alpha_{\max}(x)$.
\end{lemma}

\begin{proof}
Given $\sup_{x\in E}\alpha_{\min}(x)<\alpha<\inf_{x\in E}\alpha_{\max}(x)$, we define $q_\iii$ and $\tau_\iii$ using this $\alpha$ and define a probability measure $\nu$
on $X$ with $\spt(\nu)=E$ by setting
\begin{equation*}
\nu(E_{\iii i})=(p_{\iii}^i)^{q_\iii} r_{\iii i}^{-\tau_\iii}\nu(E_\iii)
\end{equation*}
for $\iii\in\Sigma_*$ and $i\in\{1,\ldots,m\}$.
Recall that $\sum_{i=1}^m (p_{\iii}^i)^{q_\iii} r_{\iii i}^{-\tau_\iii}=1$ by Theorem \ref{thm:localdimappl} and Lemma \ref{taulemmaq<0}.

The condition (M8) implies that
\begin{align}
\ldimloc(\nu,x_\iii)&=\liminf_{n\rightarrow\infty}\log\nu(E_{\iii|_n})/\log\diam(E_{\iii|_n}),\label{ojan2}\\
  \udimloc(\nu,x_\iii)&=\limsup_{n\rightarrow\infty}\log\nu(E_{\iii|_n})/\log\diam(E_{\iii|_n}) \label{ojan}
\end{align}
for all $\iii\in \Sigma$, and similar formulas apply for $\mu$. Also, using (M6) and Lemma \ref{prop:tekninen}, we have
\begin{align}\label{eq:cornudeff}
\frac{\log\nu(E_\iii)}{\log\left(\mu(E_\iii)^{q_\iii}\diam(E_\iii)^{-\tau_{\iii}}\right)}\longrightarrow 1,
\end{align}
uniformly as $|\iii|\rightarrow\infty$.
Let $\eta,\delta>0$. 
Then there is $n_0\in\N$ so that if $n\ge n_0$, we have
\begin{align*}
\nu\bigl( \{ x_\iii \in E :\; &\mu(E_{\iii|_n}) < \diam(E_{\iii|_n})^{\alpha + 3\eta} \} \bigr) \\
 & = \nu\bigl( \{ x_\iii \in E : \mu(E_{\iii|_n})^{-\delta} \diam(E_{\iii|_n})^{\delta(\alpha + 3\eta)} \ge 1 \} \bigr) \\
&\le \sum_{\iii\in\Sigma_n}\mu(E_\iii)^{-\delta}\diam(E_\iii)^{\delta(\alpha+3\eta)}\nu(E_\iii)
\le\sum_{\iii\in\Sigma_n}\mu_{\iii}^{q_\iii-\delta} r_{\iii}^{\delta(\alpha+\eta)-\tau_\iii}.
\end{align*}
In the last estimate we
used (M6) to conclude that
$\diam(E_\iii)^{\delta(\alpha+3\eta)}<r_{\iii}^{\delta(\alpha+2\eta)}$
for all $\iii\in\Sigma_n$ and (M6) and \eqref{eq:cornudeff} to guarantee that $\nu(E_\iii)\le\mu_{\iii}^{q_{\iii}}r_{\iii}^{-\tau_\iii-\delta\eta}$. Combining this with \eqref{eq:gammaexists} and choosing $\delta>0$ small enough, we conclude that for $n\ge n_0$,
\[\nu\bigl( \{ x_\iii \in E :\; \mu(E_{\iii|_n}) < \diam(E_{\iii|_n})^{\alpha + 3\eta} \} \bigr)\le C\gamma^n,\]
where $\gamma<1$ is independent of $n$.

Summing the above estimate over all $n\ge n_0$, using the Borel-Cantelli lemma, the analogs of \eqref{ojan2} and \eqref{ojan} for $\mu$ and
letting $\eta\downarrow0$, this
implies that
$\udimloc(\mu,x)\le\alpha$
for $\nu$-almost
all $x\in X$.
A similar calculation (in particular using \eqref{eq:gammaexists2} in place of \eqref{eq:gammaexists}) gives
$\ldimloc(\mu,x)\ge\alpha$
for $\nu$-almost all $x$. Thus, in particular, we have
\begin{equation}\label{tays}
\nu(X\setminus E_{\alpha})=0.
\end{equation}

From \eqref{eq:cornudeff}, Lemma \ref{prop:tekninen} and (M6), it follows that
\[\liminf_{n\rightarrow\infty}\frac{\log\nu(E_{\iii|_n})}{\log\diam(E_{\iii|_n})}=q(x_\iii)\liminf_{n\to\infty}\frac{\log\mu(E_{\iii|_n})}{\log\diam(E_{\iii|_n})}-\tau_{q(x_\iii)}(\mu,x_\iii)\]
for all $\iii\in\Sigma$ and similarly for $\limsup$.
Taking \eqref{ojan2}--\eqref{ojan} into account yields
$\dimloc(\nu,x)=q(x)\alpha-\tau_{q(x)}(\mu,x)=f(\alpha,x)$
for all $x\in E_\alpha$.
Together with \eqref{tays},
these estimates readily imply that
$\underline{f}(\alpha)\le\dimh(E_{\alpha})\le\dimp(E_{\alpha})\le
\overline{f}(\alpha)$.
\end{proof}

\begin{proof}[Proof of Theorem \ref{thm:localmultifractal}]
Let $x=x_\iii\in E$ and $\alpha_{\min}(x)\le\alpha\le\alpha_{\max}(x)$. We first consider the case $\alpha\neq\{\alpha_{\min}(x),\alpha_{\max}(x)\}$.

For each small $r>0$, let $n(r)$ and $N(r)$ be the largest and smallest natural numbers such that $B(x,r)\subset E_{\iii|_{n(r)}}$ and $E_{\iii|_{N(r)}}\subset B(x,r)$, respectively. By (M7) these are well defined for all small $r>0$ and moreover, $n(r),N(r)\longrightarrow\infty$
 as $r\downarrow 0$. Let $\alpha_0(r)=\inf\{f(\alpha,y)\,:\,y\in E_{\iii|_{N(r)}}\}$, $\alpha_1(r)=\sup\{f(\alpha,y)\,:\,y\in E_{\iii|_{n(r)}}\}$. Recall that by Lemma \ref{prop:tekninen}, if
$r>0$ is small enough then  $\sup\{\alpha_{\min}(y)\,:\,y\in E_{\iii|_{n(r)}}\}<\alpha<\inf\{\alpha_{\max}(y)\,:\,y\in E_{\iii|_{N(r)}}\}$ thanks to the assumption $\alpha_{\min}(x)<\alpha<\alpha_{\max}(x)$. In particular, $f(\alpha,y)$ is well defined for all $y\in E_{\iii|_{n(r)}}$.

Lemma \ref{lemma:spectrum} applied to $E_{\iii|_{n(r)}}$ and $E_{\iii|_{N(r)}}$ yields the estimates
\begin{align*}
 \alpha_0(r)& \le\dimh(\{y\in B(x,r)\,:\,\dimloc(\mu,y)=\alpha\})\\
            & \le\dimp(\{y\in B(x,r)\,:\,\dimloc(\mu,y)=\alpha\})\le\alpha_1(r).
\end{align*}
From Lemma \ref{prop:tekninen} we infer that $\alpha_0(r), \alpha_1(r) \longrightarrow f(\alpha,x)$ as $r\downarrow0$ and this gives the claim.

Finally, let us assume that $\alpha=\alpha_{\min}(x)$ (the case $\alpha=\alpha_{\max}(x)$ is symmetric). In the degenerate case $\alpha_{\min}(x)=\alpha_{\max}(x)$, we have $f(\alpha,x)=\dimloc(\mu,x)=\alpha$ and the claim follows using \eqref{eq:min-max}.
If $\alpha_{\min}(x)<\alpha_{\max}(x)$, then $f(\alpha_{\min}(x),x)=0$. Given $\eps>0$, we may consider the set
\[E_{\eps,r}=\{y\in B(x,r)\,:\,\udimloc(\mu,x)\le\alpha_{\min}(x)+\eps\}.\]
Now a minor variation of (the proof of) Lemma \ref{lemma:spectrum} implies that there is a constant $c<\infty$ (independent of $r$ and $\eps$) such that for all small $r>0$, we have $\dimp(E_{\eps,r})\le c\eps$. Letting $\eps\downarrow 0$ finishes the proof. 
\end{proof}

\begin{remark}
  Our aim in this paper was to present a simple situation where local multifractal analysis could be carried out in general metric spaces. For this reason, we assumed that the measure $\mu$ locally resembles a Bernoulli measure. Concerning possible generalisations, it is natural to ask if Theorem \ref{thm:localdimappl} (or a version of it) remains true when the measure $\mu$ is required to locally resemble a quasi-Bernoulli measure. For example, see \cite{BrownMichonPeyriere1992, Feng2007} for such results in the global setting.
\end{remark}

\subsection*{Acknowledgments}
We are grateful to the referees for carefully reading the manuscript and for their comments that lead to several improvements.

\end{document}